\documentclass[11pt,letterpaper]{amsart}
\usepackage{amsmath}
\usepackage{amstext}
\usepackage{amssymb}
\usepackage{amsfonts}
\usepackage{enumerate}
\usepackage{mathrsfs}
\usepackage{color}
\usepackage[all]{xy}

\numberwithin{equation}{section}

\newtheoremstyle{note}% <name>
{1em}% <Space above>
{1em}% <Space below>
{}% <Body font>
{}% <Indent amount>
{\bfseries}% <Theorem head font>
{:}% <Punctuation after theorem head>
{.5em}% <Space after theorem headi>
{}% <Theorem head spec (can be left empty, meaning `normal')>

\newtheorem{theorem}{Theorem}[section]
\newtheorem{lemma}[theorem]{Lemma}
\newtheorem{proposition}[theorem]{Proposition}
\newtheorem{corollary}[theorem]{Corollary}
\newtheorem{definition}[theorem]{Definition}

\theoremstyle{note}

\newtheorem{question}[theorem]{Question}

\newcommand{\N}{{\mathbb{N}}}

\newcommand{\K}{{\mathbb{K}}}
\newcommand{\n}[1]{ \left\|#1\right\| }

\DeclareMathOperator{\tr}{tr}

\DeclareMathOperator{\spn}{span}
\DeclareMathOperator{\rank}{rank}

\title{Frame potential for finite-dimensional Banach spaces}

\author{J.A. Ch\'{a}vez-Dom\'{i}nguez}
\address{Department of Mathematics\\ University of Oklahoma\\ Norman OK , 73019 USA}
\email{jachavezd@ou.edu}

\author{ D. Freeman}
\address{Department of Mathematics and Statistics\\
St Louis University\\
St Louis, MO 63103  USA} \email{daniel.freeman@slu.edu}

\author{K. Kornelson}
\address{Department of Mathematics\\ University of Oklahoma\\ Norman OK , 73019 USA}
\email{kkornelson@ou.edu}

\thanks{
The first author was supported by NSF grant DMS-1400588, and the second author was supported by grant 353293 from the Simon's foundation.}


\begin{document}

\begin{abstract}
We define the frame potential for a Schauder frame on a finite dimensional Banach space as the square of the $2$-summing norm of the frame operator.  As is the case for frames for Hilbert spaces, we prove that the frame potential can be used to characterize finite unit norm tight frames (FUNTFs) for finite dimensional Banach spaces.  We prove the existence of FUNTFs for a variety of spaces, and in particular that every $n$-dimensional complex Banach space with a $1$-unconditional basis has a FUNTF of $N$ vectors for every $N\geq n$.  However, many interesting results on FUNTFs and sums of rank one projections for Hilbert spaces remain unknown for Banach spaces and we conclude the paper with multiple open questions.  
\end{abstract}
\maketitle

\section{Introduction}

A collection of vectors $(x_j)_{j\in I}$ in a Hilbert space $H$ is called a {\em frame} if there exist constants $A,B>0$, called the {\em frame bounds}, such that 
\begin{equation}\label{E:def_frame}
A\|x\|^2\leq \sum_{j\in I} |\langle x,x_j\rangle|^2\leq B\|x\|^2\quad\quad\textrm{ for all }x\in H.
\end{equation}
The frame is called \emph{tight} if the lower frame bound $A$ equals the upper frame bound $B$.  Inequality \eqref{E:def_frame} can be interpreted as stating that the map $\Theta:H\rightarrow \ell_2(I)$ given by $\Theta(x)=(\langle x,x_j\rangle)_{j\in I}$  is an isomorphic embedding.   This operator allows us to define a collection of vectors $(f_j)_{j\in I}$ by $f_j=(\Theta^*\Theta)^{-1/2}x_j$ for all $j\in I$.  The sequence $(f_j)_{j\in I}$ is itself a frame for $H$,  called the {\em canonical dual frame} of $(x_j)_{j\in I}$, and it gives the following reconstruction formula.
\begin{equation}\label{E:def_frame2}
x = \sum_{j\in I} \langle x,f_j\rangle x_j \quad\quad\textrm{ for all }x\in H.
\end{equation}
Thus, a frame allows for a continuous linear reconstruction formula for all vectors in a Hilbert space.  We think of frames as possibly redundant coordinate systems  in the sense that the vectors $(f_j)_{j\in J}$ used for reconstruction using a frame may not be unique.  This redundancy can be useful in applications as if some coefficients of a basis are lost, this results in the loss of entire dimensions, but if some frame coefficients are lost, the error can be distributed over the whole space.  This concept motivates understanding which frames are most resilient to certain kinds of error \cite{Holmes-Paulsen,GKK}.
A {\em finite unit norm tight frame (FUNTF)}, is a tight frame of unit vectors for a finite dimensional Hilbert space.  FUNTFs are of particular interest, as they minimize both the mean squared error due to noise and the reconstruction error due to the loss of a single coefficient \cite{GKK}.   All orthonormal bases for a Hilbert space are equivalent, but there is a wide variety of different FUNTFs and their geometry carries multiple interesting properties.  

Significant interest in FUNTFs started in 2003 when Benedetto and Fickus proved that for every $k\geq n$, every $n$-dimensional Hilbert space has a FUNTF of $k$ vectors \cite{BF03}.  To do this they defined a positive function on the set $\{(x_j)_{j=1}^k:x_j\in H, \|x_j\|=1\}$, called the {\em frame potential}, so that a collection of vectors minimizes the frame potential if and only if the vectors form a tight frame for $H$.  Since then, researchers have been interested in understanding the geometry and topology of FUNTFs \cite{BH15}\cite{CMS17} and properties of the frame potential itself \cite{FJKO05}\cite{JO08}.  In \cite{DFKLOW}, an algorithm is given to create FUNTFs and in \cite{CF09} it is shown that the method of gradient descent on the frame potential can be used to create FUNTFs.    

 Frames have been generalized to Banach space in many similar but distinct ways, such as atomic decompositions \cite{FG88}, Banach frames \cite{G91}, framings \cite{CHL99}, and Schauder frames \cite{CDOSZ08}.  We will be specifically considering Schauder frames as they are a direct generalization of \eqref{E:def_frame2} without any additional assumptions. Given a Banach space $X$ with dual $X^*$, a \emph{Schauder frame} $(x_j,f_j)_{j=1}^n\subseteq X\times X^*$ (where $n\in\N$ or $n=\infty$) is a sequence of pairs such that we have the following reconstruction formula
\begin{equation}\label{E:def_Sframe}
x = \sum_{j=1}^n f_j(x) x_j \quad\quad\textrm{ for all }x\in X.
\end{equation}
If $(x_j)_{j=1}^\infty$ is a Schauder basis for a Banach space $X$ with biorthogonal functionals $(x_j^*)_{j=1}^\infty$ then $x=\sum_{j=1}^\infty x_j^*(x) x_j$ for all $x\in X$.  Thus, $(x_j,x_j^*)_{j=1}^\infty$ is a Schauder frame.  This gives that Schauder frames are both a generalization of frames for Hilbert spaces and Schauder bases for Banach spaces.  It is interesting to consider both what properties of frames for Hilbert spaces extend to Schauder frames, and what properties of Schauder bases extend to Schauder frames.  For example, frames for Hilbert spaces can be characterized as projections of Riesz bases for Hilbert spaces \cite{HL00}, and Schauder frames can be characterized as projections of Schauder bases \cite{CDOSZ08}.  Furthermore, many of the fundamental characterizations for shrinking and boundedly complete Schauder bases extend nicely to shrinking and boundedly complete Schauder frames \cite{CL09,Liu10,CLS11,BFL15}. Our goal is to extend the well-developed theory and applications for finite unit norm tight frames and the frame potential to the general finite dimensional Banach space setting.  In particular, we will prove that $(x_j,f_j)_{j=1}^N$ is a FUNTF for an $n$-dimensional Banach space if and only if the 2-summing norm of the frame operator of $(x_j,f_j)_{j=1}^N$ is $\frac{N}{\sqrt{n}}$.  
  Though FUNTFs are relatively new mathematical objects, we show that their generalization to Banach spaces is directly connected to some classical notions in the geometry of Banach spaces.

\section{Preliminaries}

Let $X$ be a finite dimensional Banach space with dual $X^*$.  Given a sequence of pairs $(x_j,f_j)_{j=1}^N\subseteq X\times X^*$, the {\em frame operator} of $(x_j,f_j)_{j=1}^N$ is the map $S:X\rightarrow X$ defined by $S(x)=\sum_{j=1}^N f_j(x) x_j$ \cite{FOSZ14}.  The sequence $(x_j,f_j)_{j=1}^N$ is called an {\em approximate Schauder frame} if the frame operator is bounded and invertible \cite{FOSZ14}.  Note that a Schauder frame is an approximate Schauder frame whose frame operator is the identity. 
\begin{definition}
Let $X$ be a finite dimensional Banach space and $(x_j,x^*_j)_{j=1}^N\subseteq X\times X^*$ such that $\|x_j\|=x^*_j(x_j)=\|x_j\|=1$ for all $1\leq j\leq N$.  We say that $(x_j,x^*_j)_{j=1}^N$ is a {\em finite unit norm tight frame} (FUNTF) for $X$ if its frame operator is a scalar multiple of the identity.
\end{definition}
 Note that in the case that $X$ is a Hilbert space, $(x_j,f_j)_{j=1}^N$ is a FUNTF in our definition if and only if $x_j=f_j$ for all $1\leq j\leq N$ and $(x_j)_{j=1}^N$ is a FUNTF in the original Hilbert space definition.   In the case that $X$ is an $N$-dimensional Banach space, then $(x_j,x^*_j)_{j=1}^N$ is a FUNTF for $X$ if and only if $(x_j)_{j=1}^N$ is an Auerbach basis for $X$ with biorthogonal functionals $(x_j^*)_{j=1}^N$.

Given a frame $(x_j)_{j=1}^N$ for a finite dimensional Hilbert space, the {\em Frame Potential} of $(x_j)_{j=1}^N$ is the value
\begin{equation}\label{E:HilbertFP}
FP((x_j)_{j=1}^N)=\sum_{i,j=1}^N|\langle x_j, x_i\rangle|^2.
\end{equation}
The frame potential can also be calculated as the square of the Hilbert-Schmidt norm of the frame operator.  The Hilbert-Schmidt norm  is only defined for operators on Hilbert spaces, but the 2-summing norm is defined for operators on any finite dimensional Banach space and the 2-summing norm of an operator on a Hilbert space is equal to its Hilbert-Schmidt norm.  This naturally allows us to extend the definition of frame potential for frames of finite dimensional Hilbert spaces to Schauder frames of finite dimensional Banach spaces.   Given $(x_j, f_j)_{j=1}^N\subseteq X\times X^*$, we define the {\em frame potential} of $(x_j, f_j)_{j=1}^N$ as the square of the 2-summing norm of the frame operator of $(x_j, f_j)_{j=1}^N$.  

Let $(X,\|\cdot\|_X)$ and $(Y,\|\cdot\|_Y)$ be Banach spaces. The {\em 2-summing norm} of an operator $T:X\rightarrow Y$ is the value $\pi_2(T)$ which satisfies:
\begin{equation}\label{E:2sum}
\pi_2(T)^2=\sup\left\{\sum_{j=1}^n\|T x_i\|_Y^2\,:\,\sum_{j=1}^n |f(x_j)|^2\leq \|f\|^2_{X^*} \textrm{ for all } n\in\N, f\in X^*,(x_j)_{j=1}^n\subseteq X\right\}.
\end{equation}

 The following theorem shows that the frame potential can be used to characterize FUNTFs for Hilbert spaces.
\begin{theorem}[\cite{BF03}]\label{T:BFmain}
Let $(x_j)_{j=1}^N$ be a sequence of unit vectors in an $n$-dimensional Hilbert space with $n\leq N$.  Then, the frame potential of $(x_j)_{j=1}^N$ is at least $\frac{N^2}{n}$, and $(x_j)_{j=1}^N$ is a FUNTF if and only if the frame potential of $(x_j)_{j=1}^N$ is equal to $\frac{N^2}{n}$.
\end{theorem}
One of our main results is proving that Theorem \ref{T:BFmain} is true for Schauder frames as well.  The following result follows from Theorem \ref{thm-pi2-as-frame-potential} which we prove in Section \ref{S:2}.

\begin{theorem}\label{T:main}
Let $X$ be an $n$-dimensional Banach space, $N \ge n$,  $x_1, \dotsc, x_N  \in X$ and $x^*_1, \dotsc x^*_N \in X^*$ such that $\|x_j\|_X=\|x_j^*\|_{X^*}=x^*_j(x_j) = 1$ for $1 \le j \le N$.  Then, the frame potential of $(x_j,x_j^*)_{j=1}^N$ is at least $\frac{N^2}{n}$, and $(x_j,x_j^*)_{j=1}^N$ is a FUNTF if and only if the frame potential of $(x_j)_{j=1}^N$ is equal to $\frac{N^2}{n}$.
\end{theorem}
It can be difficult to calculate the frame potential using the 2-summing norm, but there is fortunately extensive literature on the subject.  Before getting into properties of the 2-summing norm and our definition of the frame potential, we prove that the natural alternatives for an explicit simple frame potential on Banach spaces don't work.
 Our definition for a frame potential uses a generalization of the Hilbert-Schmidt norm, but it is conceivable that the explicit formula  given in \eqref{E:HilbertFP} could be generalized instead.  Given an approximate Schauder frame $(x_j,x_j^*)_{j=1}^N\subseteq X^*\times X$, the two natural candidates for an alternative  formula for a Schauder frame potential are the values 
$\sum_{j=1}^N\sum_{i=1}^N |x^*_j(x_i)|^2$ and $\sum_{j=1}^N\sum_{i=1}^N |x^*_j(x_i) x^*_i(x_j)| $.  However, neither of these are appropriate formulas for a frame potential because as the following proposition shows, Theorem \ref{T:main} would be false for both formulas.  Thus, although the 2-summing norm can be difficult to calculate, it works for defining a frame potential and the simple alternative formulas do not. 

\begin{proposition}\label{P:frameFail}
Let $F_3(\ell_1^2)$ be the set of length 3 approximate Schauder frames $(x_j,x_j^*)_{j=1}^3$ of $\ell_1^2$ such that $\|x_j\|_1=\|x_j^*\|_\infty=x^*_j(x_j)=1$ for $j=1,2,3$.  
There exists a FUNTF $(x_j,x_j^*)_{j=1}^3\in F_3(\ell_1^2)$ and  $(y_j,y_j^*)_{j=1}^3,(z_j,z_j^*)_{j=1}^3\in F_3(\ell_1^2)$  so that $(y_j,y^*_j)_{j=1}^3$ and $(z_j,z_j^*)_{j=1}^3$ are not FUNTFs, and the following two inequalities are satisfied.
\begin{enumerate}
\item $\sum_{j=1}^3 \sum_{k=1}^3 |x^*_j(x_k)|^2 > \sum_{j=1}^3 \sum_{k=1}^3 |y^*_j(y_k)|^2$
\item $\sum_{j=1}^3 \sum_{k=1}^3 |x^*_j(x_k) x^*_k(x_j)| > \sum_{j=1}^3 \sum_{k=1}^3 |z^*_j(z_k) z^*_k(z_j)| $
\end{enumerate}
\end{proposition}

\begin{proof}
Let $(e_1,e_2)$ be the unit vector basis for $\ell_1^2$ with biorthogonal functionals $(e_1^*,e_2^*)$.    We will define a FUNTF $(x_j,x^*_j)_{j=1}^3$ for $\ell_1^2$ by 
 $$x_1=e_1,\, x_2=\frac{1}{4}e_1-\frac{3}{4}e_2,\,  x_3=\frac{1}{4}e_1+\frac{3}{4}e_2,\quad  x^*_1=e^*_1,\, x^*_2=e^*_1-e^*_2,\,  x^*_3=e^*_1+e^*_2.$$ 
Then we have that $\|x_j\|_{1}=\|x^*_j\|_\infty=x_j^*(x_j)=1$ for all $j=1,2,3$.  Furthermore, it is simple to check that for all $x\in \ell_1^2$, we have that $\sum_{j=1}^3 x^*_j(x) x_j=\frac{3}{2} x$.  Thus, $(x_j,x^*_j)_{j=1}^3$ is a FUNTF for $\ell_1^2$.  We now define $(y_j,y^*_j)_{j=1}^3$
 by 
 $$y_1=e_1,\, y_2=\frac{1}{2}e_1-\frac{1}{2}e_2,\,  y_3=\frac{1}{2}e_1+\frac{1}{2}e_2,\quad  y^*_1=e^*_1,\, y^*_2=e^*_1-e^*_2,\,  y^*_3=e^*_1+e^*_2.$$ 
Then we have that $\|y_j\|_{1}=\|y^*_j\|_\infty=y_j^*(y_j)=1$ for all $j=1,2,3$.  However, $\sum_{j=1}^3 y^*_j(e_1) y_j= 2 e_1$ and $\sum_{j=1}^3 y^*_j(e_2) y_j= e_2$.  Thus, $(y_j,y^*_j)_{j=1}^3$ is not a FUNTF.  

A direct calculation shows that,
 $$\sum_{j=1}^3 \sum_{k=1}^3 |x^*_j(x_k)|^2= 5+\frac{5}{8}> 5+\frac{1}{2}= \sum_{j=1}^3 \sum_{k=1}^3 |y^*_j(y_k)|^2.$$
Which proves (1).  We now define $(z_j,z^*_j)_{j=1}^3$
 by 
 $$z_1=e_1,\, z_2=e_2,\,  z_3=\frac{1}{2}e_1+\frac{1}{2}e_2,\quad  z^*_1=e^*_1-e^*_2,\, z^*_2=e^*_2,\,  z^*_3=e^*_1+e^*_2.$$ 
Then we have that $\|z_j\|_{1}=\|z^*_j\|_\infty=z_j^*(z_j)=1$ for all $j=1,2,3$.  However, $\sum_{j=1}^3 z^*_j(e_1) z_j=\frac{3}{2} e_1+\frac{1}{2}e_2$.  Thus, $(z_j,z^*_j)_{j=1}^3$ is not a FUNTF.

A direct calculation proves (2) as, 
$$\sum_{j=1}^3 \sum_{k=1}^3 |x^*_j(x_k)x^*_k(x_j)|= 4+\frac{1}{2}> 4= \sum_{j=1}^3 \sum_{k=1}^3 |z^*_j(z_k)z^*_k(z_j)|.$$
\end{proof}

\section{Properties of the $2$-summing norm and the frame potential}\label{S:2}

Recall that if $(X,\|\cdot\|_X)$ and $(Y,\|\cdot\|_Y)$ are Banach spaces then the {\em 2-summing norm} of an operator $T:X\rightarrow Y$ is the value $\pi_2(T)$ which satisfies:
\begin{equation}\label{E:2sum2}
\pi_2(T)^2=\sup\left\{\sum_{j=1}^n\|T x_i\|_Y^2\,:\,\sum_{j=1}^n |f(x_j)|^2\leq \|f\|^2_{X^*} \textrm{ for all } n\in\N, f\in X^*,(x_j)_{j=1}^n\subseteq X\right\}.
\end{equation}
The space of all $2$-summing operators from $X$ to $Y$, equipped with the norm $\pi_2(\cdot)$, will be denoted by $\Pi_2(X,Y)$ (or simply $\Pi_2(X)$ when $Y=X$).
Equation \eqref{E:2sum2} may look daunting, but the literature contains many useful techniques and results to assist us.
Even though \eqref{E:2sum2} requires us to consider sequences of any arbitrary length $n$, when the operator $T$ has rank $k$ it suffices to consider $n=k(k+1)/2$ in the real case and $n=k^2$ in the complex case \cite[Thm. 18.2]{Tomczak}.
One result that we will use extensively is that 
the space of $2$-summing operators from a Banach space $X$ to itself is in trace duality with itself \cite[Prop. 9.10]{Tomczak}.  In particular, $f$ is a norm $1$ linear functional defined on the space of $2$-summing operators from $X$ to $X$ if and only if there is an operator $S:X\rightarrow X$ with $\pi_2(S)=1$ and $f(T)=\tr(ST)$ for every $2$-summing operator $T:X\rightarrow X$.
Another important ingredient will be the fact that if $X$ is an $n$-dimensional normed space, then the $2$-summing norm of the identity map $I_X : X \to X$ is equal to $\sqrt{n}$ \cite[Thm. 4.17]{Diestel-Jarchow-Tonge}.
We will also make use of the following characterization of $2$-summing operators (see \cite[Cor. 16.10.1]{G07} or \cite[Cor. 2.16]{Diestel-Jarchow-Tonge} for a reference).

\begin{theorem}[Pietsch factorization theorem]
Let $X$ and $Y$ be Banach spaces and $T:X\rightarrow Y$ be a linear operator.  Then $T$ is a $2$-summing operator if and only if there is a probability measure space $M$ and linear operators $A:X\rightarrow L_\infty(M)$ and $B:L_2(M)\rightarrow Y$ with $\|A\|\|B\|=\pi_2(T)$ so that the following diagram commutes, where $I_{\infty,2}$ is the formal identity from $L_\infty(M)$ to $L_2(M)$.
$$
\xymatrix{
 L_\infty(M) \ar[r]_{I_{\infty,2}}& L_2(M) \ar[d]^B \\
X \ar[u]^A \ar[r]_S &X
}
$$
\end{theorem}

Recall that a Banach space $X$ is smooth at a point $x\in X \setminus\{ 0\}$ if  there exists a unique normalizing functional of $x$ which we call $x^*$. That is, $x^* \in X^*$ with $\n{x^*} =1$ and $x^*(x) = \n{x}$.  A Banach space is called smooth if it is smooth at every non-zero point.
More generally, for a smooth Banach space we define the duality map $J : X \to X^*$ as follows: for any $x \in X$, $Jx$ is the unique functional in $X^*$ satisfying $\n{Jx} = \n{x}$ and $(Jx)(x) = \n{x}^2$.   The space of Hilbert-Schmidt operators on a Hilbert space is itself a Hilbert space, and is hence smooth.  Though the space of $2$-summing operators on a Banach space may not be smooth, we prove below that the space of $2$-summing operators on any finite dimensional Banach space is smooth at the identity map.

\begin{theorem}\label{T:smooth}
Let $X$ be a finite dimensional Banach space.  Then the space of $2$-summing operators on $X$ is smooth at the identity map $I_X$.
\end{theorem}
\begin{proof}
Let $X$ be an $n$-dimensional Banach space. 
Let $S:X\rightarrow X$ with $\pi_2(S)=1$, we have that 
\begin{equation}\label{E:smoothId}
n^{1/2}=\pi_2(I_X)\leq \tr(I_X\circ S)=\tr(S)
\end{equation}
Note that we have equality in \eqref{E:smoothId} for $S=n^{-1/2} I_X$ and hence $n^{-1/2} I_X$ is a norming functional for $I_X$.  We will prove that the space of $2$-summing operators is smooth at the identity operator by showing that equality occurs in \eqref{E:smoothId} if and only if $S=n^{-1/2} I_X$.  Suppose that equality occurs in \eqref{E:smoothId} for an operator $S$ on $X$ with $\pi_2(S)=1$.  In particular, $\tr(S)=n^{1/2}$.  By the Pietsch Factorization Theorem, we have that there exists a probability measure space $M$ and operators $A:X\rightarrow L_\infty(M)$ with $\|A\|=1$ and $B:L_2(M)\rightarrow Y$ with $\|B\|=\pi_2(S)=1$ so that the following diagram commutes.
$$
\xymatrix{
 L_\infty(M) \ar[r]_{I_{\infty,2}}& L_2(M) \ar[d]^B \\
X \ar[u]^A \ar[r]_S &X
}
$$
Let $H=(B^{-1}(0))^\perp$ and $P_H:L_2(M)\rightarrow H$ be the orthogonal projection.  Note that $dim(H)\leq n$ as $X$ is $n$-dimensional and hence $\pi_2(I_H)\leq n^{1/2}$. We now have the following inequality.
\begin{align*}
n^{1/2} &= \tr(S) = \tr(B|_H P_H I_{\infty,2} A) = \tr(P_H I_{\infty,2} A B|_H) = \tr( I_{H} P_H I_{\infty,2} A B|_H) \\
	&\le \pi_2(I_{H}) \pi_2(P_H I_{\infty,2} A B|_H) \le n^{1/2} \|P_H\|\pi_2( I_{\infty,2})\|A\|\|B|_H\| = n^{1/2}
\end{align*}
Therefore, all the inequalities above are in fact equalities. In particular, $\pi_2(I_H)=n^{1/2}$ and hence $B$ is rank $n$ and $B|_H:H\rightarrow X$ is invertible.  Furthermore, 
$$
\tr( I_{H} P_H I_{\infty,2} A B|_H)=\pi_2(I_{H}) \pi_2(P_H I_{\infty,2} A B|_H).
$$
Note that $H$ is an $n$-dimensional Hilbert space and the $2$-summing norm coincides with the Hilbert-Schmidt norm for operators on Hilbert spaces.  Thus, as the Hilbert-Schmidt norm is smooth, we have that $P_H I_{\infty,2} A B|_H=n^{-1/2} I_H$ is the unique normalizing functional of $I_H$.   As $B|_H$ is invertible, we have that $B|_HP_H I_{\infty,2} A B|_H (B|_H)^{-1}= B|_H n^{-1/2} I_H (B|_H)^{-1}$.  Hence, $S=n^{-1/2} I_X$.
\end{proof}

Recall that the frame potential of an approximate Schauder frame is the square of the $2$-summing norm of the frame operator. Thus, the following result shows that the frame potential can be used to characterize FUNTFs.

\begin{theorem}\label{thm-pi2-as-frame-potential}
Let $X$ be an $n$-dimensional Banach space, $N \ge n$,  $x_1, \dotsc, x_N  \in X$ and $x^*_1, \dotsc x^*_N \in X^*$ such that $x^*_j(x_j) = 1$ for $1 \le j \le N$.
Define the operator $S : X \to X$ by
$$
Sx = \sum_{j=1}^N x^*_j(x) x_j.
$$
Then
\begin{equation}\label{eqn-frame-potential-bound}
\frac{N}{\sqrt{n}} \le \pi_2(S).
\end{equation}
Moreover, equality in \eqref{eqn-frame-potential-bound} occurs if and only if $S = \tfrac{N}{n} I_X$.
\end{theorem}
\begin{proof}
We have the following inequality,
$$\pi_2(S)\geq \tr(n^{-1/2}I_X S) =n^{-1/2} \tr(\sum_{j=1}^N x^*_j\otimes x_j)= n^{-1/2}\sum_{j=1}^N x_j^*(x_j)=n^{-1/2}N
$$
Furthermore, as the $2$-summing norm is smooth at the identity operator, we have that equality occurs if and only if $S$ is a scalar multiple of the identity operator.  Thus, there exists a scalar $\lambda$ such that $S=\lambda I_X$.  By taking the trace of both sides, we get that $\lambda=\tfrac{N}{n}$.
\end{proof}

\section{When do FUNTFs of length $N$ exist?}\label{S:exists}

In Theorem \ref{thm-pi2-as-frame-potential} we have obtained a nice lower bound of $N^2/n$ for the frame potential, which is achieved only when the associated frame operator is a multiple of the identity. Nevertheless, it does not show that the bound is actually achieved.
More generally, that result does not prove the existence of FUNTFs of a given length $N$ for a general Banach space: this section 
deals with precisely that question. We start by recording the easy fact that FUNTFs always exist when their length is a multiple of the dimension of the space. 

\begin{proposition}
Let $X$ be an $n$-dimensional Banach space. If $N$ is a multiple of $n$, then there exists a FUNTF of length $N$ for $X$.
\end{proposition}

\begin{proof}
Let $(e_j, e_j^*)_{j=1}^n$ be an Auerbach basis for $X$; this is clearly a FUNTF of length $n$ for $X$.
When $N$ is a larger multiple of $n$, we simply take $N/n$ copies of the Auerbach basis.
\end{proof}

It should be mentioned that in general Auerbach bases are not unique. In a recent paper \cite{Weber-Wojciechowski}, where they prove an old conjecture of Pe{\l}czy\'{n}ski,
Weber and Wojciechowski show that a Banach space of dimension $n>2$ has at least $(n-1)n/2+1$ Auerbach bases.

Note that if $y \in X$ and $y^* \in X^*$ satisfy $1= \n{y} = \n{y^*} = y^*(y)$ then the operator $x \mapsto y^*(x) y$ is a norm one, rank one projection, and in fact this characterizes the norm one, rank one projections. Therefore, looking for FUNTFs corresponds to figuring out when a multiple of the identity can be written as a sum of norm one, rank one projections.
In the Hilbert space case, invertible operators that can be written as such sums have a very nice characterization \cite{Kornelson-Larson}: they need to be positive and have integer trace at least the dimension of the space.
%\footnote{
%It would also be interesting to consider the case of sums of \emph{Hermitian} rank one projections (see \cite{Lumer,Vidav,Berkson}),
%as well as Hermitian frames. They may have nicer properties closer to the Hilbert space case.} 
For complex Banach spaces, a related question (dropping the norm one condition) also has a very satisfactory answer: 
\cite[Thm. 4.4]{Bart-Ehrhardt-Silbermann-LogarithmicResidues} shows that an operator $T : X \to X$ is a sum of rank one projections if and only if $\tr(T)$ is an integer and $\rank(T) \le \tr(T)$.

We have not been able to obtain, for a general finite-dimensional Banach space $X$, a
characterization of the invertible operators $X \to X$ that can be written as a sum of $N$ norm one, rank one projections.
Nevertheless, we prove below that in the special case of a diagonal operator on a  complex space with a 1-unconditional basis, the ``obvious'' conditions are enough. As a consequence, we prove the existence of FUNTFs on such spaces (and a few others).
The crux of the argument is given by the following lemma, where the number of projections is equal to the dimension of the space.

\begin{lemma}\label{lemma-diagonal-as-sum-of-RONOPs}
Let $X$ be a complex $n$-dimensional Banach space (or a $2$-dimensional real space) with a normalized $1$-unconditional basis $(e_j)_{j=1}^n$ and corresponding biorthogonal functionals $(e_j^*)_{j=1}^n$.
For any sequence of nonnegative numbers $(\lambda_j)_{j=1}^n$ with $\sum_{j=1}^n \lambda_j = n$, the operator $T : X \to X$ given by
$$
T = \sum_{j=1}^n \lambda_j e_j^* \otimes e_j
$$
can be written as a sum of $n$ norm one, rank one projections.
\end{lemma}

\begin{proof}
The following argument is inspired by the discrete Fourier transform, and is related to similar constructions in the Hilbert space case
\cite{GKK,Zimmermann}.
We assume first that the space is complex.
Note that the nonnegative numbers $\lambda_j/n$ add up to one,
so by Lozanovski\u\i's factorization theorem \cite{Lozanovskii} (see also \cite{Jamison-Ruckle} for the specific finite-dimensional case we need here)
there exist sequences of nonnegative numbers $(\alpha_j)_{j=1}^n$ and $(\beta_j)_{j=1}^n$ such that
both $x = \sum_{j=1}^n \alpha_j e_j$ and $x^* = \sum_{j=1}^n \beta_j e^*_j$ have norm one, and $\alpha_j\beta_j = \lambda_j/n$ for each $1 \le j \le n$.
Let $\omega_n = e^{-2\pi i/ n}$ and
for $0 \le k \le n-1$, let
$$
x_k = \sum_{j=1}^n \omega_n^{kj} \alpha_j e_j \quad \text{and} \quad  x_k^* = \sum_{j=1}^n \omega_n^{-kj} \beta_j e_j^* 
$$
Note that both $x_k$ and $x_k^*$ have norm one by the $1$-unconditionality of $(e_j)_{j=1}^n$, and moreover $x_k^*(x_k) = \sum_{j=1}^n \alpha_j\beta_j = 1$.
Writing the map $x_k^* \otimes x_k$ as a matrix with respect to the bases $(e_j)_{j=1}^n$ and $(e^*_j)_{j=1}^n$, its entry in the $(i,j)$ position is
$$
\omega_n^{ki} \alpha_i \omega_n^{-kj} \beta_j
$$
and therefore the entry of $\sum_{k=0}^{n-1}x_k^* \otimes x_k $ in the $(i,j)$ position is
\begin{enumerate}[(a)]
\item When $i = j$:
$$
\sum_{k=0}^{n-1} \alpha_j\beta_j = n \lambda_j/n = \lambda_j
$$
\item When $i \not= j$:
$$
\alpha_i\beta_j \sum_{k=0}^{n-1} \omega_n^{k(i-j)} = 0 
$$
\end{enumerate}
That is, $\sum_{j=1}^n \lambda_j e_j^* \otimes e_j = \sum_{k=0}^{n-1}x_k^* \otimes x_k$.

A very similar argument works in the real case for $n=2$: take
$x_1 = \alpha_1 e_1 + \alpha_2 e_2$ ,  $x_2 = \alpha_1 e_1 - \alpha_2 e_2$,
$x^*_1 = \beta_1 e_1 + \beta_2 e_2$ and $x^*_2 = \beta_1 e_1 - \beta_2 e_2$;
it follows that
$$
x_1^* \otimes x_1 + x_2^* \otimes x_2 = \lambda_1 e_1^* \otimes e_1 + \lambda_2 e_2^* \otimes e_2.
$$
\end{proof}

Now we can prove the existence of FUNTFs in a wide variety of spaces.

\begin{proposition}\label{proposition-integer-trace-implies-sum-of-RONOPs}
Let $X$ be a complex $n$-dimensional Banach space (or a $2$-dimensional real space) with a normalized $1$-unconditional basis $(e_j)_{j=1}^n$ and corresponding biorthogonal functionals $(e_j^*)_{j=1}^n$.
Let $N \ge n$ be an integer, and assume the nonnegative numbers $(\lambda_j)_{j=1}^n$ satisfy $\sum_{j=1}^n \lambda_j = N$.
Then the operator $T : X \to X$ given by
$$
T = \sum_{j=1}^n \lambda_j e_j^* \otimes e_j
$$
can be written as a sum of $N$ norm one, rank one projections.
In particular, there exists a FUNTF of length $N$ for $X$.
\end{proposition}

\begin{proof}
We will proceed by induction on $N$.
If $N=n$, we're done by Lemma \ref{lemma-diagonal-as-sum-of-RONOPs}.
Now suppose the statement holds whenever the operator has trace $N$, and take a sequence of positive numbers $(\lambda_j)_{j=1}^n$ adding up to $N+1$. Note that there exists $\lambda_{j_0}$ strictly greater than one, since $n < N+1$.
Consider now the sequence $(\lambda'_j)_{j=1}^n$ where $\lambda_{j_0}$ is replaced by $\lambda_{j_0}-1$; the corresponding operator $T' = \sum_{j=1}^n \lambda'_j e_j^* \otimes e_j$  can be expressed as a sum of $N$ rank one, norm one projections; adding $e_{j_0}^* \otimes e_{j_0}$ gives a corresponding decomposition for $T = \sum_{j=1}^n \lambda_j e_j^* \otimes e_j $.
\end{proof}

%---------------------------------------------------------------------
%\section{When do we have the full picture?}

Putting together our various results, this is the most general setting where the frame potential can be used to find tight unit norm frames.

\begin{theorem}
Let $X$ be a complex $n$-dimensional Banach space (or a $2$-dimensional real space) with a normalized $1$-unconditional basis.
A sequence $(x_j,x_j^*)_{j=1}^N$ in $X \times X^*$ satisfying $x_j^*(x_j) = 1$ for each $1 \le j \le N$ 
minimizes the $2$-summing norm of its associated frame operator if and only if it is a FUNTF.
\end{theorem}

\begin{proof}
According to Theorem \ref{thm-pi2-as-frame-potential}, all we need to do is show the existence of one such sequence whose frame operator has $2$-summing norm exactly $\frac{N}{\sqrt{n}}$;
this is a consequence of Proposition \ref{proposition-integer-trace-implies-sum-of-RONOPs} applied to the operator $ \tfrac{N}{n} I_{X}$.
\end{proof}

%---------------------------------------------------------------
%---------------------------------------------------------------
%---------------------------------------------------------------
\section{Smoothness and strict convexity of $\Pi_2(X,X)$}

In Theorem \ref{T:smooth}, we proved that the space of $2$-summing operators $\Pi_2(X,X)$, is smooth at $I_X$ for any finite-dimensional Banach space $X$.
Setting our aim higher, it would be interesting to know when $\Pi_2(X,X)$ is itself smooth at every point.
Notice that $X^*$ embeds isometrically into $\Pi_2(X,X)$, so $X$ being strictly convex will be a necessary condition, as strict convexity and smoothness are dual properties.
Moreover $\Pi_2(X,X)$ is in trace duality with itself, so we may equivalently study when $\Pi_2(X,X)$ is strictly convex. 
Recall that a Banach space $Y$ is called {\em strictly convex} if whenever $x,y\in Y$ are such that $\|x\|=\|y\|=\frac12\|x+y\|$ we have that $x=y$. 
The following result gives a characterization of when the space of $2$-summing operators is strictly convex in a slightly more general setting,  and is done in terms of a unique-extension condition for $2$-summing maps on subspaces.

\begin{lemma}\label{lemma-strict-convexity-of-pi2}
Let $X$ and $Y$ be finite-dimensional Banach spaces, with $Y$ being strictly convex. 
The following conditions are equivalent:
\begin{enumerate}[(a)]
\item $\Pi_2(X,Y)$ is strictly convex.
\item For any subspace $E$ of $X$ and any linear operator $t : E \to Y$, there exists a unique linear extension $T : E \to Y$ with $\pi_2(T) = \pi_2(t)$.
\end{enumerate}
\end{lemma}

\begin{proof}
(b) $\Rightarrow$ (a):
Let $T,S : X \to Y$ be linear operators with $\pi_2(T) = \pi_2(S) =  \pi_2\big( \tfrac{1}{2}(S+T) \big) = 1$.
Let $(x_1)_{j=1}^N$ be a collection of vectors in $X$ such that
$$
\sum_{j=1}^N |f(x_j)|^2\leq \|f\|^2_{X^*} \textrm{ for all } f\in X^*
 \quad \text{ and } \quad \pi_2\big( \tfrac{1}{2}(S+T) \big) = \Big( \sum_{j=1}^N \n{\tfrac{1}{2}(S+T)x_j}^2 \Big)^{1/2}.
$$
Note that
\begin{align*}
\pi_2\big( \tfrac{1}{2}(S+T) \big) &=  \n{ (\tfrac{1}{2}\n{Sx_j+Tx_j} )_{j=1}^N}_{\ell_2} \le  \n{ (\tfrac{1}{2}\n{Sx_j} + \tfrac{1}{2}\n{Tx_j} )_{j=1}^N}_{\ell_2} \\
&=  \n{ \tfrac{1}{2}( \n{Sx_j} )_{j=1}^N + \tfrac{1}{2}(\n{Tx_j} )_{j=1}^N}_{\ell_2} \\
&\le \tfrac{1}{2} \n{ ( \n{Sx_j} )_{j=1}^N}_{\ell_2} + \tfrac{1}{2}\n{(\n{Tx_j} )_{j=1}^N}_{\ell_2} \\
&\le \tfrac{1}{2} \pi_2(S) + \tfrac{1}{2} \pi_2(T)
 \end{align*}
Since we have that $\pi_2\big( \tfrac{1}{2}(S+T) \big) = \tfrac{1}{2} \pi_2(S) + \tfrac{1}{2} \pi_2(T)$, all the inequalities above must be equalities and it follows that $Tx_j = Sx_j$ for $1 \le j \le N$.
Let $E = \spn(x_j)_{j=1}^N \subset X$, and let $v : E \to Y$ be the restriction of $T$ (or $S$) to $E$. By the choice of the sequence $(x_j)_{j=1}^N$, note that $\pi_2(v) = 1$.
By assumption, there is a unique extension of $v$ to an operator $V : X \to Y$ with $\pi_2(V) = 1$. Since both $S$ and $T$ are extensions of $v$ with $2$-summing norm equal to 1, it follows that $S=T=V$.

(a) $\Rightarrow$ (b):
We will prove the contrapositive.
Suppose there exist a subspace $E \subseteq X$ and an operator $t : E \to Y$ admitting two distinct extensions $S,T : X \to Y$ with $\pi_2(t) = \pi_2(S) = \pi_2(T)$. By homogeneity, we may assume $\pi_2(t) = 1$.
Note that $ \tfrac{1}{2}(S+T)$ is also an extension of $t$, and therefore
$$
1 = \pi_2(t) \le \pi_2\big( \tfrac{1}{2}(S+T)\big) \le \tfrac{1}{2} \pi_2(S) + \tfrac{1}{2} \pi_2(T) = 1,
$$
so $ \pi_2\big( \tfrac{1}{2}(S+T)\big) = 1$, showing that $\Pi_2(X,Y)$ is not strictly convex.
\end{proof}

It is a well-known and important property of $2$-summing maps that if $t : E \to Y$ is $2$-summing, and $E$ is a subspace of $X$, then there exists an extension $T : X \to Y$ with $\pi_2(T) = \pi_2(t)$ (this follows easily from the Pietsch factorization theorem and the $1$-injectivity of $L_\infty$ spaces, see \cite[Thm. 4.15]{Diestel-Jarchow-Tonge}).
Lemma \ref{lemma-strict-convexity-of-pi2} above shows that the uniqueness of such extensions is related to geometric properties of the space of $2$-summing maps, and is of the same nature as the following classical theorem due to Taylor \cite{Taylor-Extension} and Foguel \cite{Foguel}:

\begin{theorem}\label{thm-Taylor-Foguel}
For a normed space $X$, 
the following conditions are equivalent:
\begin{enumerate}[(a)]
\item $X^*$ is strictly convex.
\item For any subspace $E$ of $X$ and any linear functional $t : E \to \K$, there exists a unique linear extension $T : E \to \K$ with $\n{T} = \n{t}$.
\end{enumerate}
\end{theorem}

There are results related to Theorem \ref{thm-Taylor-Foguel} that give equivalent geometrical characterizations of unique extension properties (not only the general situation above, but also specializations to extensions for a given fixed subspace $E \subseteq X$ or even for a given fixed functional  $t : E \to \K$; see for example \cite{BR-nested-sequences} and \cite{Oja-Poldvere-Intersection-properties}). 
Putting together Lemma \ref{lemma-strict-convexity-of-pi2} and \cite[Thm. 3.1]{Oja-Poldvere-Intersection-properties}, we get the following:

\begin{corollary}
Let $X$ and $Y$ be finite-dimensional Banach spaces, with $Y$ being strictly convex. 
The following conditions are equivalent:
\begin{enumerate}[(a)]
\item $\Pi_2(X,Y)$ is strictly convex.
\item For any subspace $E$ of $X$ and any linear operator $t : E \to Y$, there exists a unique linear extension $T : E \to Y$ with $\pi_2(T) = \pi_2(t)$.
\item For any subspace $E$ of $X$, any $\varepsilon>0$, any $T \in \Pi_2(Y,X)$ and  any sequence $(S_n)$ in $\Pi_2(Y,E)$ with $\pi_2(S_1) \le 1$ and $\pi_2(S_{n+1} - S_n) \le 1$ for all $n\in\N$, there exist $S \in \Pi_2(Y,E)$ and $n_0 \in \N$ such that
$$
\pi_2(T - S \pm S_{n_0} ) \le n_0 + \varepsilon.
$$
\end{enumerate}
\end{corollary}

\begin{proof}
Notice that $\Pi_2(E,Y) = \big( \Pi_2(Y,E) \big)^*$ and $\Pi_2(X,Y) = \big( \Pi_2(Y,X) \big)^*$ via trace duality in both cases. 
Since $\Pi_2(Y,E)$ is isometrically contained in $\Pi_2(Y,X)$ in the obvious way, condition (b) is just a particular case of the uniqueness of extensions for linear functionals characterized in \cite[Thm. 3.1]{Oja-Poldvere-Intersection-properties}.
\end{proof}

In the infinite-dimensional case, the question of the uniform convexity of $\Pi_2(X,X)$ has been studied up to isomorphism.
Lin has shown that if $\Pi_2(X,Y)$ is $B$-convex (in particular, if it is superreflexive), then both $X$ and $Y$ have cotype $2+\varepsilon$ for any $\varepsilon>0$ \cite{Lin-Pi2}.
Additionally, if $E$ is superreflexive and has cotype 2 then $\Pi_2(\ell_2,E)$ is superreflexive as well.
This is done with an ultraproduct argument, and in fact it follows from the following: $\Pi_2(\ell_2^n,E)$ is isomorphic to a subspace of $L_2(E)$ when $E$ has cotype 2, and the Banach-Mazur distance between the two spaces is less than $2 C_2(E)$ (where $C_2(E)$ is the cotype 2 constant of $E$).
Pisier has proved a closely related result \cite{Pisier-Pi2}, namely that $\Pi_2(\ell_p,\ell_p)$ is superreflexive when $1<p<2$.
The argument uses complex interpolation, and explicitly what is shown is that $\Pi_2(\ell_p,\ell_p)$ has an equivalent norm that is strictly convex.

In the $2$-dimensional case, we can prove that one always has uniqueness of extensions preserving the $2$-summing norm.

\begin{proposition}\label{lemma-2-dimensional-implies-strict-convexity}
Let $X$ be a 2-dimensional, strictly convex and smooth space. Then $\Pi_2(X,X)$ is strictly convex. 
\end{proposition}

\begin{proof}
By Lemma \ref{lemma-strict-convexity-of-pi2}, it suffices to show the uniqueness of extensions preserving the $2$-summing norm. Since $X$ is $2$-dimensional, it suffices to consider $1$-dimensional subspaces.
So let $E \subset X$ be a $1$-dimensional subspace, and $t : E \to X$ be a non-zero operator.
Since $t$ has rank one we have $\n{t} = \pi_2(t)$.
Note that any extension $T$ of $t$ preserving the $2$-summing norm will also have $2$-summing norm equal to its norm, since
$$
\pi_2(T) \ge \n{T} \ge \n{t} = \pi_2(t).
$$

There is an easy way to construct an extension that preserves the $2$-summing norm:
since every 1-dimensional subspace is 1-complemented, there is a norm one projection $P : X \to E$.  Note that then $S = t \circ P$ is an extension of $t$, clearly $\pi_2(S) \ge \pi_2(t)$ by virtue of being an extension, and $\pi_2(S) = \pi_2(t \circ P) \le \pi_2(t) \n{P} = \pi_2(t)$ so $\pi_2(S) = \pi_2(t)$.  

We now assume that $T:X\rightarrow X$ is a different extension of $t$ such that $\pi_2(T)=1$.  Choose $f\in S_{X^*}$ and $y\in S_X$ such that $t(x)=f(x)y$ for all $x\in E$.  Choose $x_0\in f^{-1}(0)$ with $\|x_0\|=1$ and choose $x_1\in S_E$ so that $f(x_1)=1$.  We have that $T(x_0)\neq 0$ because otherwise we would have $T(ax_0+bx_1)=t(bx_1)=(t\circ P)(ax_0+bx_1)$ for all scalars $a,b$.  Let $g\in S_{X^*}$ such that $g(y)=\|y\|$.  As $X$ is uniformly smooth and $f\in S_{X^*}$ is the unique normalizing functional of $x_1$,  we have that $$\lim_{a\rightarrow0}\frac{\|x_1+a x_0\|-1}{a}=f(x_0)=0
$$ 
Thus,  $\forall \epsilon>0$, $\exists \delta_\epsilon>0$ so that $\|x_1+a x_0\|<1+a\epsilon$ for all $0<a<\delta_\epsilon$.  For the sake of contradiction we assume that $g(T(x_0))\neq0$  and without loss of generality that $g(T(x_0))>0$.  Then for $0<a<\delta_{g(T(x_0))}$ we have that
$$\|T(x_1+ax_0)\|\geq g(T(x_1)+aT(x_0))=1+ag(T(x_0))>\|x_1+ax_0\|
$$
This contradicts that $\|T\|=1$.  Thus, we must have that $g(T(x_0))=0$.  As $\pi_2(T)=1$, we have that 
for all $0<a$ that there exists $x^*_a\in S_{X^*}$ such that
$$1+a^2\|T(x_0)\|^2=\|T(x_1)\|^2+\|T(ax_0)\|^2\leq |x^*_a(x_0)|^2+a^2|x^*_a(x_1)|^2
$$
By taking the limit $a\rightarrow0$ we have that $1\leq \lim_{a\rightarrow0}|x^*_a(x_0)|\leq1$ and hence as every unit norm functional which normalizes $x_0$ in absolute value is of the form  $\epsilon f$ for some $|\epsilon|=1$, we have without loss of generality that $\lim_{a\rightarrow0}x^*_a=f$.  We now have the following for all $a>0$.
\begin{align*}
1+a^2\|T(x_0)\|^2&\leq |x^*_a(x_0)|^2+a^2|x^*_a(x_1)|^2\\
1+a^2\|T(x_0)\|^2&\leq 1+a^2|x^*_a(x_1)|^2\\
a^2\|T(x_0)\|^2&\leq a^2|x^*_a(x_1)|^2\\
\|T(x_0)\|^2&\leq |x^*_a(x_1)|^2
\end{align*}
This contradicts that $\lim_{a\rightarrow0} |x^*_a(x_1)|=|f(x_1)|=0$.

\end{proof}

If $X$ is isometric to a Hilbert space then $\Pi_2(X,X)$ is isometric to a Hilbert space as well, and is hence strictly convex.  In contrast to this, the following result shows that $\Pi_2(X,X)$ fails strict convexity when $X$ is not isometric to a Hilbert space but has a non-$1$-complemented $2$-dimensional subspace which is.

\begin{theorem}
If $\ell_2^2$ is isometric to a subspace of $X$ which is not 1-complemented in $X$, then  $\Pi_2(X,X)$ is not strictly convex.
\end{theorem}
\begin{proof}
Let $Y\subseteq X$ be isometric to $\ell_2^2$ and not be 1-complemented in $X$.  Let $(e_1,e_2)$ be an orthonormal basis for $Y$ with biorthogonal functionals $(e^*_1,e^*_2)$ in $X^*$ such that $\|e^*_1\|=\|e^*_2\|=1$.    If we consider the operator 
$P_1:X\rightarrow Y\subseteq X$ given by $P_1(x)=e^*_1(x)e_1+e^*_2(x)e_2$ then $P_1$ is a projection onto $Y$ with $\pi_2(P_1)=\sqrt{2}=\pi_2(I_Y)$.

As $Y$ is not 1-complemented, there exists $y\in X$ such that $\|y\|<\|P_1(y)\|=1$.  We now choose a new orthonormal basis $(f_1,f_2)$ for $Y$ with biorthogonal functionals $(f^*_1,f^*_2)$ in $X^*$ such that $f_1=P_1(y)$ and $\|f^*_1\|=\|f^*_2\|=1$.   
The operator 
$P_2:X\rightarrow Y\subseteq X$ given by $P_2(x)=f^*_1(x)f_1+f^*_2(x)f_2$ is a projection onto $Y$ with $\pi_2(P_2)=\sqrt{2}=\pi_2(I_Y)$.   For the sake of contradiction, we assume that $P_1=P_2$.  Then, 
$P_2(y)=P_1(y)=f_1$.  Thus we have that $f_1^*(y)=1$ and $f_2^*(y)=0$.  Hence, $\|P_2(y)\|=|f_1^*(y)|\leq\|y\|$.  This contradicts that $P_2=P_1$.  We thus have two different linear extensions of $I|_Y$ with $\pi_2(P_1)=\pi_2(P_2)=\pi_2(I|_Y)$ and hence $\Pi_2(X,X)$ is not strictly convex.
\end{proof}

\section{Optimal frames for erasures}

We will now prove a result in the spirit of Holmes and Paulsen \cite{Holmes-Paulsen}, by introducing a numerical measure
of how well a frame reconstructs vectors when one or more of the frame coefficients of a vector is lost.  Let $(x_j,x_j^*)_{j=1}^N$ be  a Schauder frame.
Fix $1 \le k_1<\cdots<k_m \le N$, and suppose that the $k_1,...,k_m$ frame coefficients (that is, the measurements corresponding to $x_{k_1}^*,...,x_{k_m}^*$) are lost.
Let $S_{[k_1,...,k_m]}$ be the frame operator associated to the situation with the lost coefficients.  That is, $S_{[k_1,...,k_m]}=\sum_{i\neq k_1,...,k_m} x_i^*\otimes x_i$.
Define the \emph{maximal erasure error} for the frame $(x_j,x_j^*)_{j=1}^N$ due to the loss of $m$ coordinates to be
$$
e_m\big( (x_j,x_j^*)_{j=1}^N  \big) = \max_{1 \le k_1<\cdots<k_m \le N} \n{ S - S_{[k_1,...,k_m]}}$$
Note that in the case of the loss of one coordinate we have that,
$$e_1\big( (x_j,x_j^*)_{j=1}^N  \big) = \max_{1 \le j \le N} \n{ S - S_{[j]}}= \max_{1 \le j \le N}\|x_j^*\otimes x_j\|=\max_{1 \le j \le N}\|x_j^*\|\|x_j\|$$

In the case when $X$ is a finite dimensional Hilbert space, the equal norm tight frames minimize the erasure error due to the loss of one coordinate \cite{GKK}, and the equiangular frames (when they exist) minimize the erasure error due to the loss of two coordinates \cite{Holmes-Paulsen}.  In the case of one erasure, we have the corresponding result for Banach spaces.

\begin{proposition}\label{P:opE}
Let $N\geq n$ and let $X$ be an $n$-dimensional Banach space such that there exists a FUNTF for $X$ of length $N$.   Suppose that $(x_j,f_j)_{j=1}^N$ is a Schauder frame for $X$.  Then the following are equivalent.
\begin{enumerate}[(a)]
\item $(x_j,f_j)_{j=1}^N$ minimizes the maximal error due to one erasure.
\item  $\|x_j\|\|f_j\|=f_j(x_j)=n/N$ for all $1\leq j\leq N$.
\item $(\frac{x_j}{\|x_j\|},\frac{f_j}{\|f_j\|})_{j=1}^N$ is a FUNTF.
\end{enumerate}
\end{proposition}

\begin{proof}
We first prove  $(b)\Rightarrow(c)$.  Suppose  $\|x_j\|\|f_j\|=f_j(x_j)=n/N$ for all $1\leq j\leq N$.  For any constant $d$ and nonzero constants $c_1,...,c_N$ we have that $(d c_j x_j, \frac{1}{c_j} f_j)_{j=1}^N$ has a frame operator of $d$ times the frame operator of $(x_j, f_j)_{j=1}^N$.  Thus, the frame operator of $(\frac{x_j}{\|x_j\|},\frac{f_j}{\|f_j\|})_{j=1}^N$ is $\frac{N}{n}I_X$.  Furthermore, $\|\frac{x_j}{\|x_j\|}\|=\|\frac{f_j}{\|f_j\|}\|=\frac{f_j}{\|f_j\|}(\frac{x_j}{\|x_j\|})=1$ for all $1\leq j\leq N$.  Thus $(\frac{x_j}{\|x_j\|},\frac{f_j}{\|f_j\|})_{j=1}^N$ is a FUNTF.  The argument can be reversed to show $(c)\Rightarrow(b)$ as any FUNTF has frame operator $\frac{N}{n} I_X$.

We now prove $(a)\Rightarrow(b)$ by contrapositive.  We assume that it is not the case that $\|x_j\|\|f_j\|=f_j(x_j)=n/N$ for all $1\leq j\leq N$.  
As, $(x_j,f_j)_{j=1}^N$ is a Schauder frame, its frame operator is $\sum_{j=1}^N f_j\otimes x_j=I_X$.  By taking the trace, we have that $\sum_{j=1}^N f_j(x_j)=n$.   Thus, $$
\sum_{j=1}^N \|x_j\|\|f_j\|\geq \sum_{j=1}^N f_j(x_j)=n.$$  Hence there exists $1\leq k\leq N$ such that $\|x_k\|\|f_k\|>\frac{n}{N}$.  We have that $\|f_k\otimes x_k\|=\|f_k\|\|x_k\|>n/N$ is the error due to the erasure of the $k$th coordinate.  However, there exists a FUNTF $(y_j,y_j^*)_{j=1}^N$ of $X$.  Thus, $(\frac{n}{N}y_j,y_j^*)_{j=1}^N$ is a Schauder frame of $X$ and $\frac{n}{N}$ is the error due to one erasure.  Thus, $(x_j,f_j)_{j=1}^N$ does not minimize the maximal error due to one erasure.

We now prove $(b)\Rightarrow(a)$.  We have previously shown that if $(x_j,f_j)_{j=1}^N$ minimizes the maximal error due to one erasure then it satisfies $(b)$.  However, any frame which satisfies $(b)$ has the same error due to one erasure of exactly $n/N$.  Thus, if $(x_j,f_j)_{j=1}^N$ satisfies $(b)$ then it minimizes the maximal error due to one erasure.
\end{proof}

\section{The case of real $\ell_1^n$}\label{S:ell1}

Theorem \ref{proposition-integer-trace-implies-sum-of-RONOPs} gives that for every $N\geq n$, complex $\ell_1^n$ has a length $N$ FUNTF.  
The factorization theorem of Lozanovski\u\i{}  used in the proof of Lemma \ref{lemma-diagonal-as-sum-of-RONOPs} suggests that understanding FUNTFs in real $\ell_1^n$  will give insight into the situation for real spaces with a normalized $1$-unconditional basis.  The following are some partial results for the case of real $\ell_1^n$, and by duality we have the same results for $\ell_\infty^n$ as well.

\begin{proposition}\label{P:n+1}
For all $n\in\N$, $\ell_1^n$  has a FUNTF of length $n+1$.
\end{proposition}
\begin{proof}
Note that we just need to consider the case $n\geq 3$ as every 2 dimensional Banach space with a symmetric basis has a FUNTF of all possible sizes.  Let $(e_j)_{j=1}^n$ be the unit vector basis for $\ell_1^n$ and $(e_j^*)_{j=1}^n$ be the biorthogonal functionals.
For $1\leq j\leq n$ and $0\leq a\leq 1$ we let 
$$x_j=a e_j-\sum_{i\neq j} (1-a)(n-1)^{-1} e_i\quad\textrm{ and }\quad x_{n+1}=\sum_{i=1}^n n^{-1} e_i.
$$
The corresponding normalizing functionals are
$$x^{*}_j= e^*_j-\sum_{i\neq j} e^*_i\quad\textrm{ and }\quad x^*_{n+1}=\sum_{i=1}^n  e^*_i.
$$
We will prove that there exists some constant $a\in(0,1)$ such that $(x_j,x_j^*)_{j=1}^{n+1}$ is a FUNTF.
For each $1\leq j\leq n+1$, it is clear that $\|x_j\|=\|x_j^*\|=x_j^*(x_j)=1$.  We now check the frame operator of $(x_j,x_j^*)_{j=1}^{n+1}$.
For $1\leq m\leq n$ we have that 
\begin{align*}
\sum_{j=1}^{n+1} x^*_j(e_m)e^*_m(x_j)&=x^*_m(e_m)e^*_m(x_m) +\! \sum_{j\neq m,n+1}\! x^*_j(e_m)e^*_m(x_j)+x^*_{n+1}(e_m)e^*_m(x_{n+1})\\
&=a + (n-1)(1-a)(n-1)^{-1}+n^{-1}=1+n^{-1}
\end{align*}

For $1\leq m,k\leq n$ with $m\neq k$ we have that 
\begin{align*}
\sum_{j=1}^{n+1} x^*_j(e_m)e^*_k(x_j)&=x^*_m(e_m)e^*_k(x_m) + x^*_k(e_m)e^*_k(x_k) +\!\! \sum_{j\neq k,m,n+1}\!\! x^*_j(e_m)e^*_k(x_j)+x^*_{n+1}(e_m)e^*_k(x_{n+1})\\
&=-(1-a)(n-1)^{-1}-a  + (n-2)(1-a)(n-1)^{-1}+n^{-1}\\
&=-a  + (n-3)(1-a)(n-1)^{-1}+n^{-1}\\
\end{align*}
This value is positive for $a=0$ and negative for $a=1$.  Thus, there exists $a\in (0,1)$ such that $\sum_{j=1}^{n+1} x^*_j(e_m)e^*_k(x_j)=0$.  This proves that the frame operator of $(x_j,x_j^*)_{j=1}^{n+1}$ is $(1+\frac{1}{n})$ times the identity.
\end{proof}

In the previous proposition, we gave a construction of a FUNTF of $n+1$ vectors in $\ell_1^n$ for $n\geq3$.  The proof only considered the case $n\geq 3$ because we already knew the result for $n=2$.  This is fortunate, because as the following proposition shows, the construction in Proposition \ref{P:n+1} actually fails for $n=2$.

\begin{proposition}
Every FUNTF of odd length  in $\ell_1^2$ includes an element of the canonical basis (up to a sign).
\end{proposition}
\begin{proof}
Suppose there is a FUNTF of length 3 in $\ell_1^2$, say consisting of vectors $x_1$, $x_2$ and $x_3$, which does not include an element of the canonical basis or their negatives.
This implies that all the coordinates of $x_1$, $x_2$ and $x_3$ are nonzero.
Replacing $x_i$ by $-x_i$ if necessary, we can assume that all three vectors have positive first coordinate. By reflecting the second coordinate, we can assume that both coordinates of $x_1$ are positive.
Therefore, we may assume
$$
x_1 = (a , 1-a), \quad x_2 = (b, \varepsilon_1 (1-b)), \quad x_3 = (c , \varepsilon_2 (1-c))
$$
with $a,b,c \in (0,1)$ and $\varepsilon_1, \varepsilon_2 = \pm 1$.
Their corresponding normalizing functionals must then be
$$
x^*_1 = (1 , 1), \quad x^*_2 = (1,\varepsilon_1), \quad x^*_3 = (1 ,\varepsilon_2)
$$
Note that
$$
x_1^* \otimes x_1 = \begin{pmatrix}
a &  1-a\\ 
a & 1-a\\
\end{pmatrix}
, \quad 
x_2^* \otimes x_2 = \begin{pmatrix}
b &  \varepsilon_1 (1-b)\\ 
\varepsilon_1 b & 1-b\\
\end{pmatrix},
\quad
x_3^* \otimes x_3 = \begin{pmatrix}
c &  \varepsilon_2 (1-c)\\ 
\varepsilon_2 c & 1-c\\
\end{pmatrix}
$$
It follows that in order to have a FUNTF, $a+b+c = 3/2$ and moreover we need to choose the signs  $\varepsilon_1, \varepsilon_2$ in such a way that
$$
a + \varepsilon_1b + \varepsilon_2 c = 0 \qquad \text{and} \qquad 1-a + \varepsilon_1 (1-b) + \varepsilon_2 (1-c) = 0.
$$
But adding both of these equations together gives $1 + \varepsilon_1 + \varepsilon_2 =0$, impossible due to parity.
This proves that there does not exist a length 3 FUNTF which does not contain one of the canonical basis vectors (up to a sign). Furthermore,
the same parity argument shows that a FUNTF for $\ell_1^2$ of any odd length must include an element of the canonical basis (up to a sign).
\end{proof}

For general $n\in\N$, we know that $\ell_1^n$ has a FUNTF of length $n$ and a FUNTF of length $n+1$.  As the union of FUNTFs is a FUNTF, in order to determine if $\ell_1^n$ has a FUNTF of all lengths at least the dimension we would just need to find FUNTFs of lengths $n+2,n+3,...,2n-1$.  The following proposition checks the remaining cases for dimensions 3 and 4.  Thus, $\ell_1^2$, $\ell_1^3$, and $\ell_1^4$ all have FUNTFs of all lengths at least their dimension.

\begin{proposition}
\begin{enumerate}[(a)]
\item
There exists a FUNTF of length $5$ in $\ell_1^3$.
\item There exists a FUNTF of length $6$ in $\ell_1^4$.
\item There exists a FUNTF of length $7$ in $\ell_1^4$.
\end{enumerate}
\end{proposition}

\begin{proof}
(a) We would like to use a collection of vectors
$$
x_1 = (1, 0, 0),
x_2 = (-a, b ,b ),
x_3 = (-a, -b ,b ),
x_4 = (-a, b ,-b ),
x_5 = (-a, -b ,-b ).
$$
with $a,b>0$ and $a + 2b = 1$.
(our choice of $-a$ instead of $a$ is so that the frame looks more like a ``pyramid'' in $\ell_1^3$).
The corresponding normalizing functionals are
$$
x^*_1 = (1, 0, 0),
x^*_2 = (-1, 1 ,1 ),
x^*_3 = (-1, -1 ,1 ),
x^*_4 = (-1, 1 ,-1 ),
x^*_5 = (-1, -1 ,-1 ).
$$
A calculation shows that 
$$
\sum_{j=1}^5 x^*_j \otimes x_j = \begin{pmatrix}
1+4a &  0 & 0\\ 
0 &4b & 0\\
0 &0 &4b
\end{pmatrix},
$$
so choosing $a = 1/6$ and $b = 5/12$ works.

(b) For $a,,b,c,d>0$ with $a+2b=c+d=1$, take 
\begin{multline*}
x_1 = (a, b ,b ,0),
x_2 = (a, -b ,b ,0),
x_3 = (a, b ,-b ,0), \\
x_4 = (a, -b ,-b ,0),
x_5 = (c, 0, 0, d), 
x_6 = (c,0,0,-d).
\end{multline*}
and
\begin{multline*}
x^*_1 = (1, 1 ,1 ,0),
x^*_2 = (1, -1 ,1 ,0),
x^*_3 = (1, 1 ,-1 ,0), \\
x^*_4 = (1, -1 ,-1 ,0),
x^*_5 = (1, 0, 0, 1), 
x^*_6 = (1,0,0,-1).
\end{multline*}
A calculation shows that 
$$
\sum_{j=1}^6 x^*_j \otimes x_j = \begin{pmatrix}
4a+2c &  0 & 0 &0\\ 
0 &4b & 0 &0 \\
0 &0 &4b &0 \\
0 &0 &0 &2d
\end{pmatrix},
$$
We then choose $a=1/4$, $b=3/8$, $c=1/4$, $d=3/4$.

(c) The argument is quite similar: choose vectors as in the previous example with $a=1/8$, $b=7/16$, $c=5/8$ and $d=3/8$, together with
$$
x_7 = (0,0,0,1), \qquad x^*_7 = (0,0,0,1).
$$
\end{proof}

\section{Open problems}
 We showed in Section \ref{S:exists} that a large class of finite dimensional Banach spaces have finite unit norm tight frames of every length at least the dimension of the space.  However, we do not have any examples where this is not possible.

\begin{question}
Does every $n$ dimensional Banach space have a length $N$ FUNTF  for all $N\geq n$?
\end{question}

It seems very difficult to create a method of constructing  FUNTFs of arbitrary size that works in any finite dimensional Banach space.  Thus, it may be best to focus first on specific classical Banach spaces.  We have shown in Section \ref{S:exists} that complex $\ell_1^n$ has a FUNTF of length $N$ for all $n\leq N$ and in Section \ref{S:ell1} that real $\ell_1^2$,  $\ell_1^3$, and $\ell_1^4$  each have FUNTFs of all lengths at least their dimension.  

\begin{question}
Does real $\ell_1^n$ have a length $N$ FUNTF for all $N\geq n$?
\end{question}

Note that for each $n$, there are only finitely many values of $N$ for which we do not know the answer to the question above. Indeed, it follows from our results that for any $N \ge n(n-1)$ a FUNTF of length $N$ does exist: if $N=n(n-1+m)+k$ with $0 \le k \le n-1$ and $m\ge 0$, we can take the union of $k$ copies of a FUNTF of length $n+1$ and $n-1-k+m$ copies of a FUNTF of length $n$.
 
A finite dimensional Banach space $X$ has a length $N$ FUNTF if and only if a scalar multiple of the identity operator on $X$ can be expressed as a sum of $N$ normalized rank $1$ projections.    In the Hilbert space case, operators which can be written as sums of norm one rank one projections are characterized as positive operators with integer trace \cite{Kornelson-Larson}.  Operators on complex Banach spaces  which can be written as sums of rank one projections (dropping the norm one condition) can be characterized as well
\cite{Bart-Ehrhardt-Silbermann-LogarithmicResidues}.

\begin{question}
Given a finite dimensional Banach space $X$, how can we characterize what operators may be expressed as sums of norm one rank one projections ?
\end{question}

If $X$ is not strictly convex then there exist rank one operators on $X$ where $\Pi_2(X,X)$ is not smooth, as $X^*$ is isometric to a subspace of rank one operators on $\Pi_2(X,X)$.  However, we do not have any example of a finite dimensional Banach space $X$ where $\Pi_2(X,X)$ is not smooth  at an invertible operator. 
We used that  $\Pi_2(X,X)$ is smooth at the identity to prove Theorem \ref{T:main}, which characterized FUNTFs in terms of the frame potential.  
 Knowing that $\Pi_2(X,X)$ was smooth at some invertible operator $T$ would be useful in studying approximate Schauder frames for $X$ whose frame operator is normed by $T$.

\begin{question}
Let $X$ be a finite dimensional Banach space.  Is $\Pi_2(X,X)$ smooth at every invertible operator on $X$?
\end{question}

It is clear that for all $n>1$,  $\Pi_2(\ell_2^n,\ell_2^n)$ is smooth, and that $\Pi_2(\ell_1^n,\ell_1^n)$ and $\Pi_2(\ell_\infty^n, \ell_\infty^n)$ are not smooth.  However, we do not have any results for other values of $p$. 

%% Hello Alejandro!  -Nikki + Wiggles %%

\begin{question}
Let $n\in\N$ and $1<p<\infty$.  Is $\Pi_2(\ell_p^n,\ell_p^n)$ smooth?
\end{question}

In Proposition \ref{P:opE} we characterized the Schauder frames which minimize the reconstruction error due to the erasure of one coefficient as rescalings of FUNTFs.
It would be quite interesting to say something about optimality when two frame coefficients are erased. In Hilbert spaces, the equiangular frames (when they exist) are optimal under two erasures.  It would be quite interesting to find a generalization of equiangular to certain Banach spaces.
This would require control of the norms of the operators
$$
x_i^* \otimes x_i + x_j^* \otimes x_j,\quad \textrm{ for } i\neq j.
$$

%I suppose there is some chance one could interpret this in terms of a general notion of ``angles'' (see \cite{Berkson-metrics}).
%Either Hilbert space is the best possible case, or one can do better by using a weird space. I would bet it's the former.

\begin{question}
What are some examples of $N\in\N$ and finite dimensional Banach spaces where we may characterize the length $N$ Schauder frames which minimize the maximal error due to two erasures?
\end{question}

\bibliographystyle{amsalpha}

\def\cprime{$'$} \def\cprime{$'$}
\providecommand{\bysame}{\leavevmode\hbox to3em{\hrulefill}\thinspace}
\providecommand{\MR}{\relax\ifhmode\unskip\space\fi MR }
\providecommand{\MRhref}[2]{%
  \href{http://www.ams.org/mathscinet-getitem?mr=#1}{#2}
}
\providecommand{\href}[2]{#2}

\end{document}